\documentclass[12pt]{amsart}
\usepackage[top=1.5in, bottom=1.5in, left=1.4in, right=1.4in]{geometry}  
\usepackage{geometry}                
\usepackage{graphicx}
\usepackage{amssymb}
\usepackage{epstopdf}
\usepackage{color}
\usepackage{url}
\usepackage{verbatim}
\usepackage{hyperref}
\usepackage{enumitem}

\usepackage{xcolor}

\definecolor{darkbrown}{rgb}{.5,.1,.1} 

\title[Critical Groups of SRGs]{A Note on the Critical Groups of Strongly Regular Graphs and Their Generalizations}

\author{Kenneth Hung}
\address{Kenneth Hung: Meta, Menlo Park, CA, USA}
\email{me@kenhung.me}

\author{Chi Ho Yuen}
\address{Chi Ho Yuen:
Department of Mathematics, University of Oslo, Oslo, Norway
}
\email{chihy@math.uio.no}

\def\K{\mathcal{K}}

\numberwithin{equation}{section}
\theoremstyle{definition}

\newtheorem{theorem}{Theorem}[section]
\newtheorem{lemma}[theorem]{Lemma}
\newtheorem{definition}[theorem]{Definition}

\newtheorem{remark}{Remark}
\newtheorem{proposition}[theorem]{Proposition}
\newtheorem{example}[theorem]{Example}

\newcommand{\cn}{\operatorname{cn_\pm}}

\newcommand{\coker}{\operatorname{coker}}
\newcommand{\row}{\operatorname{row}}

\begin{document}

\begin{abstract}
We determine the maximum order of an element in the critical group of a strongly regular graph, and show that it achieves the spectral bound due to Lorenzini.
We extend the result to all graphs with exactly two non-zero Laplacian eigenvalues, and study the signed graph version of the problem.
We also study the monodromy pairing on the critical groups, and suggest an approach to study the structure of these groups using the pairing.
\end{abstract}

\maketitle

\vspace{-5mm}
\section{Introduction}

The {\em critical group} $\K(G)$ (also known as sandpile group, Jacobian) of a graph $G$ is a finite abelian group whose cardinality equals the number of spanning trees of $G$.
It is an interesting algebraic invariant with connections to many fields, including combinatorics (notably the {\em chip-firing game}), tropical and arithmetic  geometry, and probability.
We recommend \cite{SandpileBook,ChipBook} to the reader for more on these connections.

As the critical group can be defined as the torsion part of the cokernel of the graph Laplacian, it is natural to inquire about the relation between the spectrum of the Laplacian and the critical group.
For example, a direct corollary of Kirchhoff's Matrix--Tree Theorem is that $|\K(G)|$ can be deduced from the spectrum.
In general, while the Laplacian spectrum of a graph does not determine its critical group, it does provide extra information.
In particular, Lorenzini proved the following:

\begin{theorem} \cite[Proposition~2.6]{Lorenzini} \label{thm:Lorenzini}
Let $M\neq 0$ be an $n\times n$ diagonalizable integer matrix and let $\theta_1,\ldots,\theta_t$ be the distinct non-zero eigenvalues of $M$. Then every torsion element of $\coker M$ is killed by $\prod \theta_i$.
\end{theorem}

Intuitively, the theorem is more powerful when $M$ only has few distinct eigenvalues.
In the case of graph Laplacians, the only graphs with a unique non-zero eigenvalue are the complete graphs, but the family of graphs with exactly two distinct non-zero eigenvalues is already very interesting and includes {\em strongly regular graphs} as prominent examples, as well as other graphs from design theory \cite{DH_2LE}.
We are able to show that Lorenzini's bound is tight for these graphs with a small list of exceptions.

\begin{theorem} \label{thm:main}
Let $G$ be a graph that has exactly two distinct non-zero Laplacian eigenvalues and is neither a complete bipartite graph $K_{m,m}$ ($m\geq 2$) nor a star $K_{1,p}$ ($p\geq 2$).
Then the exponent of $\K(G)$ is exactly the product of the two distinct non-zero eigenvalues; the exponents of $\K(G)$ in the two exceptional cases differ from the spectral bound by a factor of $2$ and $p+1$, respectively.
\end{theorem}

Unlike many other works in this direction (see Section~\ref{sec:review} for a survey of related works) that are based on purely (linear) algebraic arguments, we make use of the regularities of these graphs in a combinatorial way and produce explicit elements of the group that achieve the bound. This may help future combinatorial exploration on this topic.

Indeed, as a by-product of the proof of Theorem~\ref{thm:main}, we provide a concise description of the {\em monodromy pairing} on a critical group of a strongly regular graph in Section~\ref{sec:pairing}.
We further outline an approach to relate the structure of the graph with the (local) structure of its critical group via this pairing, and prove some preliminary results along this direction.

We also study the parallel problem for {\em signed graphs}, which are graphs with signs attached to their edges.
They were introduced by Harary \cite{harary}, and further developed by Zaslavsky and others \cite{zaslavsky,zaslavsky2}.
Signed graphs have their versions of Laplacians and critical groups, hence we can also ask whether the bound in Theorem~\ref{thm:Lorenzini} is tight for particular signed graphs.
We prove a result analogous to Theorem~\ref{thm:main} for signed graphs.

\begin{theorem} \label{thm:SG}
Let $G_\sigma$ be an unbalanced signed graph with exactly two distinct Laplacian eigenvalues.
Then the exponent of $\K(G_\sigma)$ is exactly the product of the two eigenvalues.
\end{theorem}

These signed graphs generalize {\em regular two-graphs}, and are of special interest recently for their applications in the proof of the sensitivity conjecture \cite{Huang} and construction of line systems in Euclidean space \cite{Stanic2}.
Moreover, a family of decorated graphs known as {\em Adinkras}, introduced by physicists to encode special supersymmetry algebras and Clifford algebras (or representations thereof) \cite{rA,Iga_Clifford}, can be shown to have exactly two distinct Laplacian eigenvalues as well \cite{IKKY}.

\subsection{Related Works} \label{sec:review}

Lorenzini used strongly regular graphs to illustrate Theorem~\ref{thm:Lorenzini} in \cite{Lorenzini}; he also established lower bounds for the exponent in terms of the Laplacian spectrum \cite[Proposition~2.11]{Lorenzini}, which in general does not match the upper bound even in the setting of this paper.
Ducey et al. \cite{DDEMPSV} were able to determine the $p$-Sylow subgroups of the critical groups of strongly regular graphs when the eigenvalues are {\em integral} and satisfy one of several $p$-divisibility conditions.
They also described several inequalities involving $e_i$'s, the number of $\mathbb{Z}/p^i\mathbb{Z}$ summands in the primary decomposition of $\K(G)$.
See Example~\ref{ex:Cleb} for a comparison between these works and ours.

We also note that, besides results that solely depend on the parameters of strongly regular graphs, there are many other works that study specific subclasses of strongly regular graphs \cite{CSX_Paley, DGW_Rook, TS_Polar}.

Zaslavsky formulated a Matrix--Tree Theorem for signed graphs that provides a combinatorial meaning of $|\K(G_\sigma)|$\cite{zaslavsky}; the tropical interpretation of these groups was studied by Len and Zakharov \cite{LZ_Prym}.
There are fewer works on the structure of critical groups of signed groups in the literature compared with ordinary graphs.
Reiner and Tseng related critical groups of signed graphs to those of ordinary graphs via a short exact sequence, and computed some examples using it \cite{RT_Signed}.
Recently, the second author and his collaborators studied the critical groups of Adinkras \cite{IKKY}.
They observed that, following from a stronger result that was proven using the extra structure of Adinkras, the bound in Theorem~\ref{thm:Lorenzini} is tight for Adinkras.
A motivation for proving Theorem~\ref{thm:SG} is to extend this observation to its natural generality.

\section{Preliminaries} \label{sec:prelim}

Unless otherwise specified, all graphs and signed graphs are finite, simple, and connected.
The number of vertices of a graph is denoted by $n$, and we often identify the vertex set of the graph with $\{1,\ldots,n\}$; we denote by ${\bf e}_1,\ldots,{\bf e}_n$ the standard basis of $\mathbb{Z}^n$.
For two sets $A,B$, denote by $A\triangle B$ the symmetric difference of the sets, which consists of elements that belong to exactly one of the sets.

\begin{definition}
The {\em Laplacian} of a graph $G$ is the matrix $L:=D-A$, where $D$ is the diagonal matrix whose entries are the vertex degrees, and $A$ is the adjacency matrix. It is easy to see that $L$ is symmetric (hence diagonalizable) and has a simple eigenvalue 0.

The {\em critical group} $\K(G)$ of the graph is the torsion part of the cokernel of $L$ over ${\mathbb Z}$, or equivalently, the quotient group $\{{\bf u}\in\mathbb{Z}^n:\sum_i {\bf u}_i = 0\}/\row_{\mathbb Z} L$.
\end{definition}

We recall the definition of strongly regular graphs and some of their basic properties.

\begin{definition} \label{def:SRG}
A {\em $(n,k,\lambda,\mu)$-strongly regular graph} (SRG) is a $k$-regular graph in which any two adjacent (respectively, non-adjacent) vertices have exactly $\lambda$ (respectively, $\mu$) common neighbors.
\end{definition}

\begin{proposition} \cite[Chapter~21]{ACIC} \label{prop:SRG_para}
If $G$ is a $(n,k,\lambda,\mu)$-SRG, then $(n-k-1)\mu=k(k-\lambda-1)$.
If this quantity is 0, then $G$ is the complete graph $K_n$ (or $pK_m$, the disjoint union of $p$ copies of $K_m$, if we allow disconnected graphs).
If the quantity is non-zero but for some edge $uv$, every vertex is adjacent to at least one of $u,v$, then $G$ is a complete multipartite graph $K_{m,m,\ldots,m}$ with equal parts.\footnote{In most literature, these trivial cases are excluded from the definition of SRGs, but in order to have a complete classification in Theorem~\ref{thm:main}, we include them here.}
\end{proposition}

\begin{proposition} \cite[Section~3]{Lorenzini} \label{prop:SRG_eigen}
The Laplacian of a non-complete $(n,k,\lambda,\mu)$-SRG has exactly two distinct non-zero eigenvalues, whose product is $n\mu$.
\end{proposition}

Finally, we define the essential objects surrounding signed graphs.

\begin{definition}
A {\em signed graph} $G_{\sigma}$ is a graph $G=(V,E)$ together with an assignment $\sigma:E\rightarrow\{\pm\}$.
{\em Switching} a vertex flips the signs of the edges incident to it.
A signed graph is {\em balanced} if every cycle has an even number of negative edges; a signed graph is balanced if and only if it can be transformed to a graph with all edges positive by switchings \cite[Corollary~3.3]{zaslavsky}.

For two vertices $u\neq v$, their {\em net number} $\cn(u,v)$ of common neighbors is the number of positive length 2 paths between $u,v$ (both edges are of the same sign) minus the number of negative length 2 paths (edges are of opposite signs).
\end{definition}

\begin{definition}
The {\em Laplacian} of $G_{\sigma}$ is $L_{\sigma}:=D-A_{\sigma}$, where $D$ is the same diagonal matrix as the underlying $G$'s, and $A_{\sigma}$ is the signed adjacency matrix of $G_{\sigma}$, in which a positive (respectively, negative) edge $uv$ is represented by $A_{uv}=A_{vu}=1$ (respectively, $A_{uv}=A_{vu}=-1$).
$L_{\sigma}$ is of full rank if and only if $G_\sigma$ is unbalanced \cite[Proposition~9.9]{RT_Signed}, in which case the {\em critical group} $\K(G_{\sigma})$ is the cokernel of $L_{\sigma}$, or equivalently, $\mathbb{Z}^n/\row_{\mathbb Z} L$.
\end{definition}

It is straightforward to check that switchings do not change the Laplacian spectrum or the critical group.

\section{Proofs} \label{sec:result}

The proofs in various settings are similar but with a few minor differences.
Instead of describing a convoluted universal proof, we first present the proof for SRGs in details, and explain the straightforward modifications in other cases.

\subsection{Strongly Regular Graphs}

\begin{proof} [Proof of Theorem~\ref{thm:main} for SRGs]
Pick an arbitrary edge $uv$, we claim that the vector $n\mu[{\bf e}_u-{\bf e}_v]\in\{{\bf u}\in\mathbb{Z}^n:\sum_i {\bf u}_i = 0\}$ can be written as 
\begin{equation} \label{eq:SRG_decomp}
(k+\mu-\lambda-1)L_u-(k+\mu-\lambda-1)L_v+\sum_{w\in N(u)\setminus\{v\}}L_{w}-\sum_{w\in N(v)\setminus\{u\}}L_w,
\end{equation}
here $L_x$ is the $x$-th row of $L$ and $N(x)$ is the neighborhood of $x$.

We prove the claim by classifying the vertices into four types.

Case I: $x=u$ ($x=v$ is similar).
The four terms in (\ref{eq:SRG_decomp}) contribute $(k+\mu-\lambda-1)k,(k+\mu-\lambda-1), -(k-1), \lambda$ at the $x$-coordinate, respectively.
The sum of these terms is \\
\begin{align*}
~& (k+\mu-\lambda-1)k+(k+\mu-\lambda-1)-(k-1)+\lambda\\
=~&  k^2+k\mu-k\lambda+\mu-k\\
=~& k(k-\lambda-1)+(k+1)\mu\\
=~& n\mu,
\end{align*}
by Proposition~\ref{prop:SRG_para}.

Case II: $x\neq v$ is adjacent to $u$ but not to $v$ (the opposite case is similar).
The four terms in (\ref{eq:SRG_decomp}) contribute $-(k+\mu-\lambda-1), 0, -(\lambda-k), \mu-1$ at the $x$-coordinate, respectively, and sum to $0$.
The third term is $-(\lambda-k)$ because $x\in N(u)\setminus\{v\}$ implies that $L_x$ contributes $k$, while there are $\lambda$ more neighbors of $x$ from $N(u)\setminus\{v\}$ and such $L_y$ each contributes $-1$.

Case III: $x$ is adjacent to both $u,v$.
The four terms in (\ref{eq:SRG_decomp}) contribute $-(k+\mu-\lambda-1), (k+\mu-\lambda-1), -(\lambda-k-1), (\lambda-k-1)$ at the $x$-coordinate, respectively, and sum to $0$.

Case IV: $x$ is adjacent to neither of $u,v$.
The four terms in (\ref{eq:SRG_decomp}) contribute $0,0, -\mu, \mu$ at the $x$-coordinate, respectively, and sum to $0$.\\

If $G$ is not a complete bipartite graph, then there exists a vertex $w'$ such that the coefficient of $L_{w'}$ in (\ref{eq:SRG_decomp}) is zero:
if $G$ is not completely multipartite, then by Proposition~\ref{prop:SRG_para}, there exists a vertex that is not adjacent to $u$ or $v$, hence not involved in (\ref{eq:SRG_decomp});
otherwise if $G$ is completely $r$-partite for some $r\geq 3$ with partition $U_1\sqcup\ldots\sqcup U_r$, then without loss of generality we may assume $u\in U_1,v\in U_2$, pick an arbitrary $w'\in U_3$, the term $L_{w'}$ appear in (\ref{eq:SRG_decomp}) twice with opposite coefficients, hence it vanishes in the sum.

Now $\mathcal{B}_{w'}:=\{L_x: x\neq w'\}$ is a basis of $\row_{\mathbb Z} L$ in which (\ref{eq:SRG_decomp}) is expressed.
Moreover, since there exists some $w''\neq u,v$ that is adjacent to exactly one of $u,v$ (otherwise $k=\lambda+1$ and the quantity in Proposition~\ref{prop:SRG_para} is zero), the coefficient of $L_{w''}$ in (\ref{eq:SRG_decomp}) is $\pm 1$.
In particular, the gcd of coefficients is $1$ and $n\mu[{\bf e}_u-{\bf e}_v]$ cannot be an integral multiple of any element in $\row_{\mathbb Z} L$ other than itself (or its negation). That is, the order of ${\bf e}_u-{\bf e}_v$ is exactly $n\mu$.

The only case remaining is $K_{m,m}$ ($m\geq 2$), whose critical group is $(\mathbb{Z}/m\mathbb{Z})^{2m-4}\oplus(\mathbb{Z}/m^2\mathbb{Z})$ \cite{Lorenzini2}, and the exponent is a half of $n\mu=2m^2$.
Nevertheless, this fact can be seen by analyzing the proof above:
let $U_1\sqcup U_2$ be the bipartition of vertices, pick an edge $uv$ with $u\in U_1,v\in U_2$ pick $y\in U_1\setminus\{u\}$ and use the fact that $L_y = -\sum_{x\neq y} L_x$, (\ref{eq:SRG_decomp}) becomes $2mL_{u_1}-(2m-2)L_{u_2}+2\sum_{w\in U_2\setminus\{u_2\}}L_{w}$, whose gcd of coefficients is $2$, i.e., the order is $n\mu/2=m^2$.
\end{proof}

\begin{remark} \label{rm:summary}
Summarizing the proof, we need to (1) find a suitable edge $uv$ from the graph; (2) write an integral multiple of ${\bf e}_u-{\bf e}_v$ explicitly as an integral sum of $L$'s rows that only involves $u,v$ and their neighbors; (3) find a vertex $w'$ whose row $L_{w'}$ is not involved in the sum, such as when $w'$ is not $u,v$ or their neighbors; and (4) find a vertex $w''$ whose coefficient in the sum is $\pm 1$.
\end{remark}

\begin{remark}
As $G$ is connected, $\K(G)\cong\{{\bf u}\in\mathbb{Z}^n:\sum_i {\bf u}_i = 0\}/\row_{\mathbb Z} L$ is generated by the equivalence classes of ${\bf e}_u-{\bf e}_v$'s, $uv\in E$, so the above proof verifies Theorem~\ref{thm:Lorenzini} for SRGs directly.
It is possible to complete the proofs in the remaining sections to perform a similar verification, but the extra arguments are omitted in the interest of brevity.
\end{remark}

\subsection{Non-regular Graphs with Two Laplacian Eigenvalues}

The following is a summary of \cite[Section~2]{DH_2LE}, which shows that a graph with exactly two distinct non-zero Laplacian eigenvalues has some form of regularity even if it is not degree-regular.

\begin{proposition} \label{prop:2LE}
(1) A non-regular graph $G$ has exactly two distinct non-zero Laplacian eigenvalues $\theta_1,\theta_2$ if and only if there exist constants $\mu,\overline{\mu}$ such that any two non-adjacent vertices of $G$ have exactly $\mu$ common neighbors, and any two adjacent vertices of $G$ have exactly $\overline{\mu}$ common non-neighbors.

(2) The degree of a vertex of $G$ only takes one of two values $k_1,k_2$, and the number of common neighbors of two adjacent vertices is
$$
\begin{cases} 
\mu-1+k_1-k_2 &\mbox{if both vertices have degree\ } k_1,\\
\mu-1 & \mbox{if the vertices have different degrees,}\\
\mu-1+k_2-k_1 &\mbox{if both vertices have degree\ } k_2.
\end{cases}
$$

(3) These parameters satisfy the relations $\theta_1+\theta_2=k_1+k_2+1=n+\mu-\overline{\mu}$, $\theta_1\theta_2=k_1k_2+\mu=n\mu$.
\end{proposition}

The following lemma extends Proposition~\ref{prop:SRG_para}.
Recall that the {\em sum} of two disjoint graphs $G_1=(V_1,E_1), G_2=(V_2, E_2)$ is the graph $(V_1\cup V_2, E_1\cup E_2\cup (V_1\times V_2))$, i.e., two vertices are adjacent if they are from the same $G_i$ and were adjacent there, or if they are from different $G_i$'s.

\begin{lemma} \label{lem:2LE_no_w}
Let $G$ be a graph with exactly two distinct non-zero Laplacian eigenvalues (hence with structural parameters $\mu,\overline{\mu}$) and two distinct vertex degrees $k_1<k_2$.
Let $uv$ be an edge between two vertices of different degrees.
If every other vertex is adjacent to at least one of $u,v$, then $G$ is the sum of a complete graph $K_{m'}$ ($m'\geq 1$), and either a complete multipartite graph $K_{m,m,\ldots,m}$ ($m>1$) or a disjoint union of complete graphs $pK_m$ ($m\geq 1,p\geq 2$).
\end{lemma}

\begin{proof} 
We have $k_1+k_2-\mu+1=n$ as $u,v$ have exactly $\mu-1$ common neighbors, thus $k_1,k_2$ must satisfy $k_1+k_2=n+\mu-1, k_1k_2=(n-1)\mu$.
This means that the two degrees in $G$ are $k_1=\mu,k_2=n-1$.
Let $V_1,V_2$ be the collection of vertices with degrees $k_1,k_2$, respectively, and denote their sizes by $n_1,n_2$.
The induced subgraph $G_1:=G[V_1]$ is then a $(n_1,\mu-n_2,2\mu-n-n_2,\mu-n_2)$-SRG, and by Proposition~\ref{prop:SRG_para}, $G_1$ is either $K_{m,\ldots,m}$ or $pK_m$ ($p>1$, or else $G$ itself is a complete graph),  the conclusion follows.
\end{proof}

\begin{proof} [Proof of Theorem~\ref{thm:main}]
Since regular graphs with exactly two distinct non-zero Laplacian eigenvalues are precisely the SRGs \cite[Lemma~10.2.1]{AGT}, without loss of generality, we may assume the graph has two distinct vertex degrees $k_1,k_2$.
Pick an edge $uv$ such that the degrees of $u,v$ are $k_1,k_2$, respectively (such an edge exists as $G$ is connected).
Then we can verify that $n\mu[{\bf e}_u-{\bf e}_v]$ can be written as 
\begin{equation} \label{eq:2LE_decomp}
k_2L_u-k_1L_v+\sum_{w\in N(u)\setminus\{v\}}L_{w}-\sum_{w\in N(v)\setminus\{u\}}L_w,
\end{equation}
in a manner similar to the calculation for SRGs.

If $G$ is not one of the graphs in Lemma~\ref{lem:2LE_no_w}, then there exists $w'$ that is adjacent to neither of $u,v$.
Otherwise, $G$ is the sum of some $H:=K_{m'}$ and another graph $H'$, without loss of generality, we may assume $v\in H, u\in H'$.
\begin{itemize}
\item For the sum of  $K_{m'}$ and $K_{m,\ldots,m}$ (with partition $U_1\sqcup\ldots\sqcup U_r$, $r\geq 2$), without loss of generality, we may assume $u\in U_1$. Any vertex $w'$ in $U_2$ is a common neighbor of $u,v$ and the coefficient of $L_{w'}$ in (\ref{eq:2LE_decomp}) is zero.
\item For the sum of $K_{m'}$ with $m'\geq 2$ and $pK_m$, we can choose any vertex in $K_{m'}$ other than $v$ itself to be $w'$.
\item For the sum of $K_1$ and $pK_m$  with $m\geq 2$, we can choose any vertex in the same copy of $K_m$ as $u$ to be $w'$.
\item The only remaining case $K_1+ pK_1=K_{1,p}$ ($p>1$) is a tree, so $\K(K_{1,p})$ is a trivial group, whose size is obviously less than $n\mu=p+1$.
\end{itemize}

The existence of $w''\neq u,v$ that is adjacent to exactly one of $u,v$ is trivial because $N(u),N(v)$ are of different sizes.
Hence, every ingredient described in Remark~\ref{rm:summary} is available in this setting and the proof remains valid.
\end{proof}

\subsection{Signed Graphs}

We omit the $\sigma$ subscript of the signed graph Laplacian for clarity.
We only consider unbalanced signed graphs, as the balanced case is equivalent to that of ordinary graphs by switchings.
In particular, the $n$ rows of $L$ together form a basis of $\row_{\mathbb Z} L$, and we need not find the special vertex $w'$ as in Remark~\ref{rm:summary}.

It turns out that signed graphs with exactly two distinct Laplacian eigenvalues (necessarily non-zero) act similarly as their unsigned counterparts.
The next proposition summarizes some basic facts.
The regular case is \cite[Theorem~4.1]{Stanic3}, the non-regular case is \cite[Lemma~3.4]{HouTangWang2}.

\begin{proposition} \label{prop:SG_para}
If a signed graph $G_\sigma$ has exactly two distinct Laplacian eigenvalues $\theta_1,\theta_2$, then either:\\
(1) $G$ is $k$-regular, and by denoting $\lambda:=2k-\theta_1-\theta_2$, $\theta_1\theta_2=k(k-\lambda-1)$, and for every edge $uv$ of sign $c$, $\cn(u,v)=c\lambda$, while for every pair of non-adjacent $u,v$, $\cn(u,v)=0$, or\\
(2) the degree of a vertex of $G$ only takes one of two values $k_1,k_2$, and the net number of common neighbors of two vertices is $c(d+d'-\theta_1-\theta_2)$ if they are adjacent along an edge of sign $c$ and are of degrees $d,d'$, respectively, and $\cn(u,v)=0$ otherwise.
Moreover, these parameters satisfy the relations $\theta_1+\theta_2=k_1+k_2+1,\theta_1\theta_2=k_1k_2$.
\end{proposition}

The following very simple lemma serves a similar purpose as Lemma~\ref{lem:2LE_no_w} to isolate the special case of signed complete graphs: they are exceptional because ${\bf e}_u-{\bf e}_v$'s order is not maximal for any edge $uv$.

\begin{lemma} \label{lem:complete}
If a graph $G$ on $n\geq 3$ vertices has the property that $N(u)\setminus\{v\}=N(v)\setminus\{u\}$ for every edge $uv$, then $G$ is the complete graph $K_n$.
\end{lemma}

\begin{proof}
Pick an edge $uv$.
If not every vertex is adjacent to at least one of $u,v$, pick a vertex $y\not\in N(u)\cup N(v)$ that is adjacent to $x\in N(u)\cup N(v)$ (which exists by connectivity), then applying the assumption on $ux$ implies that $y\in N(x)\setminus\{u\}=N(u)\setminus\{x\}$, a contradiction.
So it remains to prove that any two vertices $x,y\in N(u)\setminus\{v\}$ are adjacent, which follows from applying the assumption on $ux$ and concluding that $y\in N(u)\setminus\{x\}=N(x)\setminus\{u\}$.
\end{proof}

\begin{proof} [Proof of Theorem~\ref{thm:SG}]
Case I: $G$ is $k$-regular but not complete.
Pick an arbitrary edge $uv$, by switching if necessary we may assume $\sigma(uv)=+$.
It can be checked that $k(k-\lambda-1)[{\bf e}_u-{\bf e}_v]$ can be written as
\begin{equation} \label{eq:SRSG_decomp}
(k-\lambda-1)L_u-(k-\lambda-1)L_v+\sum_{w\in N(u)\setminus\{v\}}\sigma(uw)L_{w}-\sum_{w\in N(v)\setminus\{u\}}\sigma(vw)L_w.
\end{equation}
We verify the claim at $u$ as an illustration of how the assumption on $\cn$'s comes into the proof.
The first three terms of (\ref{eq:SRSG_decomp}) contribute $(k-\lambda-1)k,k-\lambda-1,-(k-1)$ at the $u$-coordinate, respectively.
For every $w\in N(u)\cap N(v)$, the sign of the $u$-coordinate of $L_w$ is $-\sigma(uw)$, so the contribution of $-\sigma(vw)L_w$ at the $u$-coordinate is $\sigma(uw)\sigma(wv)$, which is the sign of the path $u-w-v$, thus the total contribution from the last term of (\ref{eq:SRSG_decomp}) is by definition $\cn(u,v)=\lambda$.
Therefore, the sum of all four terms is $(k-\lambda-1)k$ as expected.

By Lemma~\ref{lem:complete}, we may choose $uv$ in a way that $(N(u)\setminus\{v\})\triangle(N(v)\setminus\{u\})$ is non-empty.
For any $w''$ in the symmetric difference, the coefficient of $L_{w''}$ in (\ref{eq:SRSG_decomp}) is $\pm 1$, which shows that the order of ${\bf e}_u-{\bf e}_v$ is $k(k-\lambda-1)$.\\

Case II: $G$ is complete.
Up to switching, there is a unique unbalanced signing of $K_3$.
By direct computation, its Laplacian spectrum is $1,1,4$ and its critical group is isomorphic to $\mathbb{Z}/4\mathbb{Z}$, matching the theorem statement.
For $n\geq 4$, we first note that there exists an unbalanced 3-cycle $u-v-w-u$ of $G$: consider the shortest unbalanced cycle $u_1-\ldots-u_l-u_1$ of $G$, if $l>3$, then either $u_1-u_2-u_l-u_1$ or $u_2-u_3-\ldots-u_l-u_2$ is a shorter unbalanced cycle.
Without loss of generality, $uv$ is the unique negative edge on this 3-cycle.
Similar to (\ref{eq:SRSG_decomp}), $k(k-\lambda-1)[{\bf e}_u+{\bf e}_v]$ equals
\begin{equation} \label{eq:SCG_decomp}
(k-\lambda-1)L_u+(k-\lambda-1)L_v+\sum_{w\in N(u)\setminus\{v\}}\sigma(uw)L_{w}+\sum_{w\in N(v)\setminus\{u\}}\sigma(vw)L_w.
\end{equation}
Now consider $k(k-\lambda-1){\bf e}_u = \frac{k(k-\lambda-1)}{2}[({\bf e}_u+{\bf e}_v)+({\bf e}_w-{\bf e}_v)+({\bf e}_u-{\bf e}_w)]$ expressed in the rows of $L$.
Using (\ref{eq:SRSG_decomp}) and (\ref{eq:SCG_decomp}), it can be seen that the coefficient of $L_x$ for any $x\neq u,v,w$ is $\pm 1$ by a case-by-case analysis with the signs of the edges $ux,vx,wx$.
For example, if $ux$ is the unique negative edge, then the coefficients of $L_x$ in the three terms $k(k-\lambda-1)[{\bf e}_u+{\bf e}_v],k(k-\lambda-1)[{\bf e}_w-{\bf e}_v],k(k-\lambda-1)[{\bf e}_u-{\bf e}_w]$ are $0,0,-2$, respectively.
Other cases are analogous.\\

Case III: $G$ has two distinct vertex degrees $k_1,k_2$.
Pick an edge whose endpoints $u,v$ are of degrees $k_1,k_2$, respectively, and again we may assume $\sigma(uv)=+$.
$k_1k_2[{\bf e}_u-{\bf e}_v]$ can be written as
\begin{equation} \label{eq:SG2LE_decomp}
k_2L_u-k_1L_v+\sum_{w\in N(u)\setminus\{v\}}\sigma(uw)L_{w}-\sum_{w\in N(v)\setminus\{u\}}\sigma(vw)L_w.
\end{equation}
Since $N(u),N(v)$ are of different size, there exists $w''\in (N(u)\setminus\{v\})\triangle(N(v)\setminus\{u\})$ whose row's coefficient in (\ref{eq:SG2LE_decomp}) is $\pm 1$ and the conclusion follows.
\end{proof}

\section{The Monodromy Pairing on \texorpdfstring{$\K(G)$}{K(G)}} \label{sec:pairing}

The critical group of a graph $G$ is equipped with a canonical pairing $\langle \cdot,\cdot\rangle:\K(G)\times\K(G)\rightarrow \mathbb{Q}/\mathbb{Z}$ known as the {\em monodromy pairing}, which is related to Grothendieck's pairing in the theory of abelian varieties \cite{BL_Pairing} and the energy pairing in the potential theory on graphs \cite{BS_Pairing}.
We describe the monodromy pairing on the critical groups of SRGs that are not complete or complete bipartite (similar calculations can be done in other settings considered in this paper, but we restrict ourselves to SRGs for the illustration), and sketch an approach that could be useful for understanding the structure of $\K(G)$ further using the pairing.

\begin{definition} \label{def:pairing}
Let ${\bf D},{\bf D}'\in\{{\bf u}\in\mathbb{Z}^n:\sum_i {\bf u}_i = 0\}$ be two vectors representing two elements of $\K(G)$.
Choose positive integers $m,m'$ such that $L{\bf f}=m{\bf D}, L{\bf f}'=m'{\bf D}'$ for some ${\bf f},{\bf f}'\in \mathbb{Z}^n$ ($m,m'$ exist because $\K(G)$ is finite)\footnote{The notations were chosen to reflect the tropical perspective of critical groups: ${\bf D}$ stands for a {\em tropical divisor} and ${\bf f}$ stands for a {\em tropical meromorphic function}.}.
Then the {\em monodromy pairing} between $[{\bf D}],[{\bf D}']$ is $\langle [{\bf D}],[{\bf D}']\rangle:=\frac{{\bf f}^T}{m}L\frac{{\bf f}'}{m'}=\frac{{\bf f}^T{\bf D}'}{m}=\frac{{\bf D}^T{\bf f}'}{m'}\in\mathbb{Q}/\mathbb{Z}$.
\end{definition}

\begin{proposition} \cite[Lemma~1.1]{BL_Pairing} 
The pairing is well-defined, bilinear, and symmetric.
\end{proposition}

Now, by choosing $m=m'=n\mu$ in Definition~\ref{def:pairing} and using (\ref{eq:SRG_decomp}), it is easy to compute the pairing between group elements of the form $E_{uv}:=[{\bf e}_u-{\bf e}_v]$ for edges $uv$ of $G$: $\langle E_{uv}, E_{xy}\rangle$ equals the coefficient of $L_x$ minus the coefficient of $L_y$ in (\ref{eq:SRG_decomp}), divided by $n\mu$.
Notice that the said coefficients only depend on the local information of adjacency between $u,v,x,y$ (which is not obvious as SRGs are not necessarily symmetric objects in general).
Since $E_{uv}$'s generate $\K(G)$ and the pairing is bilinear, the whole pairing can be described combinatorially.
We highlight two special instances of this calculation.

\begin{proposition} \label{prop:pairing}
Let $uv$ be an edge of $G$.
Then $\langle E_{uv}, E_{uv}\rangle=\frac{2(n-1)}{kn}\neq 0\in\mathbb{Q}/\mathbb{Z}$.
Let $xy$ be another edge that shares no common vertices with $uv$, and such that either $N(x)\cap\{u,v\}=N(y)\cap\{u,v\}$, or $u,v\in N(x),u,v\not\in N(y)$, or $u,v\not\in N(x),u,v\in N(y)$.
Then $\langle E_{uv}, E_{xy}\rangle=0$.
\end{proposition}

\begin{proof}
From the above discussion, $\langle E_{uv}, E_{uv}\rangle = \frac{2(k+\mu-\lambda-1)}{n\mu}=\frac{2}{n\mu}\cdot(\frac{(n-k-1)\mu}{k}+\mu)$ by Proposition~\ref{prop:SRG_para}, which further simplifies to $\frac{2(n-1)}{kn}$.
Since we have $n,k\geq 2$, $\frac{2(n-1)}{kn}$ is a rational number strictly between $0$ and $1$.

For the second assertion, note that in (\ref{eq:SRG_decomp}), the coefficient of $L_x$ for $x\neq u,v$ is 1 if and only if $N(x)\cap\{u,v\}=\{u\}$, $-1$ if and only if $N(x)\cap\{u,v\}=\{v\}$, and $0$ if and only if $N(x)\cap\{u,v\}=\{u,v\}$ or $\emptyset$.
\end{proof}

Recall that the {\em invariant factor decomposition} of $\K(G)$ (or any finite abelian group) is the decomposition $\K(G)\cong\oplus_{i=1}^d \mathbb{Z}/n_i\mathbb{Z}$, where $n_1>1, n_1\mid n_2\mid\ldots\mid n_d$.
Using the monodromy pairing, we can give a criterion that implies the existence of a large subgroup with few generators (or equivalently, that the sequence of invariant factors is tail-heavy) from the existence of a large subset of $E_{uv}$'s orthogonal with respect to $\langle\cdot,\cdot\rangle$.
A distinction of the criterion is that it provides finer information of $\K(G)$ than merely considering the basic parameters, while only requiring local structural information of $G$.

Before stating the precise statement, we first prove an elementary lemma on finite abelian groups that we cannot find a reference for; some ideas of the proof are also being used in the main theorem of this section.
Recall that for a finite abelian group $G$ and a natural number $n$, $G[n]:=\{g\in G: n\cdot g=0\}$ is a subgroup of $G$, and when $n$ is a prime, its size is $n^l$, where $l$ is the number of summands $\mathbb{Z}/m\mathbb{Z}$ in the invariant factor (or primary) decomposition of $G$ such that $n\mid m$.

\begin{lemma} \label{lem:ab_subgp}
Let $G\cong\oplus_{i=1}^d \mathbb{Z}/n_i\mathbb{Z}$ be the invariant factor decomposition of a finite abelian group $G$, and let $H$ be a subgroup of $G$ with $r$ summands in its invariant factor decomposition.
Then $r\leq d$, and if we index the invariant factor decomposition of $H$ as $H\cong\oplus_{i=d-r+1}^d \mathbb{Z}/n'_i\mathbb{Z}$, where $n'_{d-r+1}>1,n'_{d-r+1}\mid n'_{d-r+2}\mid\ldots\mid n'_d$, then we have $n'_i\mid n_i$ for $i=d-r+1,\ldots,d$.
\end{lemma}

\begin{proof}
It suffices to prove the statement for $p$-groups as we can localize the original problem for $H\leq G$ to their $p$-Sylow subgroups.
Write $G:=\oplus_{i=1}^d \mathbb{Z}/p^{\alpha_i}\mathbb{Z}$ ($0<\alpha_1\leq\ldots\leq\alpha_d$) and $H:=\oplus_{i=d-r+1}^d \mathbb{Z}/p^{\beta_i}\mathbb{Z}$ ($0<\beta_{d-r+1}\leq\ldots\leq\beta_d$).
In the case of $r>d$, we allow non-positive indexing in the invariant factor decompositions, and denote $\alpha_{d-r+1},\ldots,\alpha_0=0$ so that $G\cong\oplus_{i=d-r+1}^d \mathbb{Z}/p^{\alpha_i}\mathbb{Z}$.

Suppose $\beta_j>\alpha_j$ for some $j\geq d-r+1$ (which evidently happen if $r>d$), consider $p^{\alpha_j}H:=\{p^{\alpha_j}\cdot h:h\in H\}\cong\oplus_{i=d-r+1}^d \mathbb{Z}/p^{\max\{\beta_i-\alpha_j,0\}}\mathbb{Z}$ and $p^{\alpha_j}G$.
The subgroup $(p^{\alpha_j}H)[p]$ has at least $p^{d-j+1}$ elements, whereas $(p^{\alpha_j}G)[p]$ has at most $p^{d-j}$ elements, a contradiction to the fact that the former is a subgroup of the latter.
\end{proof}

For technical reasons, we switch the algebraic setting in the above proof to a slightly different one for the main theorem: for a finite abelian group $G$, $G[p]$ is equicardinal to $G\otimes_{\mathbb{Z}}\mathbb{F}_p$, and the exponent (base $p$) of its size is the dimension of the tensor product as a vector space over $\mathbb{F}_p$.

Write $\frac{2(n-1)}{kn}$ in its lowest term, and let $\eta$ denote the denominator of the resulting, or equivalent, $\eta := \frac{kn}{\gcd(2(n-1), kn)}$.
Since $n\mu\frac{2(n-1)}{kn}=\langle n\mu\cdot E_{uv},E_{uv}\rangle = 0\in \mathbb{Q}/\mathbb{Z}$,  we have $\eta\mid n\mu$.
On the other hand, $\eta\neq 1$ as $\frac{2(n-1)}{kn}\not\in\mathbb{Z}$.

\begin{theorem} \label{thm:tail_heavy}
Suppose there exist $E_1:=E_{u_1v_1},\ldots,E_r:=E_{u_rv_r}\in\K(G)$ whose pairwise pairings are zero (we shall call the corresponding subset of edges {\em orthogonal}).
Then $\K(G)$ contains a subgroup isomorphic to $\mathbb{Z}/n\mu\mathbb{Z}\oplus(\mathbb{Z}/\eta\mathbb{Z})^{r-1}$.
\end{theorem}

\begin{proof}
We prove by induction that the invariant factor decomposition of $G_l:=\langle E_1,\ldots,E_l\rangle$ has exactly $l$ summands, each of size divisible by $\eta$.
The base case is trivial as $G_1\cong\mathbb{Z}/n\mu\mathbb{Z}$.
By induction hypothesis, we can write $G_{l-1}\cong\oplus_{i=2}^l \mathbb{Z}/m'_i\mathbb{Z}$ such that $\eta\mid m'_2\mid m'_3\mid\ldots\mid m'_l$.
Since the number of invariant factors is the minimum number of generators necessary to generate the group, we can write $G_l\cong\oplus_{i=1}^l \mathbb{Z}/m_i\mathbb{Z}$ such that $m_1\mid\ldots\mid m_l$ and $m_2>1$, even though we cannot rule out the possibility of $m_1=1$ for now.
By Lemma~\ref{lem:ab_subgp}, $\eta\mid m'_i\mid m_i$ for all $i\geq 2$.

Suppose $\eta\nmid m_1$.
Pick a prime $p$ dividing $\frac{\eta}{\gcd(\eta,m_1)}$, and consider $\gcd(\eta,m_1) G_l\cong\oplus_{i=1}^l \mathbb{Z}/\frac{m_i}{\gcd(\eta,m_1)}\mathbb{Z}$.
Denote by $H$ the tensor product $(\gcd(\eta,m_1) G_l)\otimes\mathbb{F}_p$, viewed as a vector space over $\mathbb{F}_p$.
On one hand, since $p\nmid \frac{m_1}{\gcd(\eta,m_1)}$ but $p\mid\frac{\eta}{\gcd(\eta,m_1)}\mid \frac{m_i}{\gcd(\eta,m_1)}$ for all $i\geq 2$, we know that $\dim H=l-1$.
On the other hand, we claim that $(\gcd(\eta,m_1)\cdot E_1)\otimes 1,\ldots, (\gcd(\eta,m_1)\cdot E_l)\otimes 1\in H$ are linearly independent over $\mathbb{F}_p$, which would be a contradiction.

Suppose $\sum t_i[(\gcd(\eta,m_1)\cdot E_i)\otimes 1]=0$ for some $t_i$'s in $\mathbb{F}_p$.
By picking arbitrary representatives in $\mathbb{Z}$ and with a slight abuse of notation, we can rewrite the equation as $[\sum t_i\gcd(\eta,m_1)\cdot E_i]\otimes 1=0$, which implies $\sum t_i\gcd(\eta,m_1)\cdot E_i = p(\sum s_i\gcd(\eta,m_1)\cdot E_i)$ for some $s_i\in\mathbb{Z}$ --- recall that $\{\gcd(\eta,m_1)\cdot E_i: i = 1, \ldots, \ell\}$ generates $\gcd(\eta,m_1) G_l$ and that the kernel of $G\rightarrow G\otimes\mathbb{F}_p$ given by $x\rightarrow x\otimes 1$ is $pG$.
Computing the pairing between $\sum (t_i-ps_i)\gcd(\eta,m_1)\cdot E_i$ and each $E_j$ yields $\frac{2(n-1)(t_j-ps_j)\gcd(\eta,m_1)}{kn}\in\mathbb{Z}$, which is only possible if $\eta\mid (t_j-ps_j)\gcd(\eta,m_1)$, i.e., $\frac{\eta}{\gcd(\eta,m_1)}\mid (t_j-ps_j)$, this in turn shows that $p\mid t_j,\forall j$ as claimed.

Finally, apply Lemma~\ref{lem:ab_subgp} again to conclude that $\eta\mid n_i, \forall i\geq d-r+1$ in the invariant factor decomposition of $\K(G)$, together with the extra knowledge that $n_d=n\mu$ produces the subgroup isomorphic to $\mathbb{Z}/n\mu\mathbb{Z}\oplus(\mathbb{Z}/\eta\mathbb{Z})^{r-1}$ as desired.
\end{proof}

In general, $\eta$ can be less than $n\mu$: according to the online database of SRGs \cite{SRG_database}, the only feasible parameter tuples for SRGs with $\eta=n\mu$ and $n\leq 100$ are $(5,2,0,1),(35,18,9,9), (45,12,3,3), (85,20,3,5)$.
Moreover, as mentioned in Section~\ref{sec:review}, there are some previously known existence results for elements of certain order in the critical groups of SRGs, in which $\eta$ can be numerically smaller than those bounds.
Nevertheless, even in the those cases, the existence of elements of order $\eta$ could still be incomparable with the known results as the following example shows.

\begin{example} \label{ex:Cleb}
The complement $\overline{\rm Cleb}$ of the {\em Clebsch graph} is a $(16,10,6,6)$-SRG constructed as follows: the vertex set consists of all binary strings of length 4, and two strings are adjacent if they differ by exactly 2 or 3 digits.
The non-zero Laplacian eigenvalues of $\overline{\rm Cleb}$ are 8 and 12, with multiplicities 5 and 10, respectively.

By \cite[Corollary~3.2]{Lorenzini}, $\K(\overline{\rm Cleb})$ contains a subgroup isomorphic to $(\mathbb{Z}/8\mathbb{Z})^4$ and a subgroup isomorphic to $(\mathbb{Z}/12\mathbb{Z})^9$, hence one can deduce (and only deduce) that $\K(\overline{\rm Cleb})$ contains an element of order 24.
The same conclusion can be obtained from applying the results in \cite{DDEMPSV}.
More precisely, as the parameters of $\overline{\rm Cleb}$ do not satisfy the conditions in their Section~3 for $p=2$, the only applicable result is their general Lemma~2.1, which is of the same nature as \cite[Corollary~3.2]{Lorenzini}.

On the other hand, the two edges of $\overline{\rm Cleb}$ connecting $(0000), (0011)$ and $(1110),(1101)$, respectively, are orthogonal from Proposition~\ref{prop:pairing}.
Theorem~\ref{thm:tail_heavy} then implies the existence of some subgroup isomorphic to $\mathbb{Z}/96\mathbb{Z}\oplus\mathbb{Z}/16\mathbb{Z}$, in which even the existence of the second summand does not follow from the known results.

For reference, the critical group of $\overline{\rm Cleb}$ is isomorphic to $(\mathbb{Z}/3\mathbb{Z})\oplus(\mathbb{Z}/12\mathbb{Z})^4\oplus(\mathbb{Z}/24\mathbb{Z})\oplus(\mathbb{Z}/96\mathbb{Z})^4$.
\end{example}

We conclude this section with some brief discussion on how to look for a large orthogonal subset, and how the problem is related to the more classical structural questions concerning SRGs.

The study of cliques in SRGs, such as bounding the clique number of SRGs, is a standard topic in the subject \cite{GS_clique}.
For example, when the SRG is coming from a {\em partial geometry} \cite[Chapter~21]{ACIC}, every line in the geometry corresponds to a clique of the SRG.
More generally, bounding the maximum size of regular induced subgraphs of a SRG also attracts some attention \cite{Evans}.

It can be seen that these special subgraphs are sources of orthogonal subsets.
For cliques, any matching within a clique (or more generally, an induced disjoint union of cliques) is orthogonal.
On the other extreme, any induced matching as a $1$-regular induced subgraph of $G$ is orthogonal.
Other than regular subgraphs, the two disjoint edges in an induced {\em paw graph} ($K_{1,3}$ plus one more edge) are also orthogonal; this construction can be extended to any matching in an induced chain of $K_3$, i.e., an induced graph consisting of $p$ copies of edge-disjoint $K_3 = \{u_i,v_i,w_i\}$'s, where $w_i=u_{i+1}$ for every $i=1,\ldots,p-1$.

\section{Concluding Remarks}

It is a natural future direction to classify graphs with few (but more than two) non-zero Laplacian eigenvalues that achieve the bound in Theorem~\ref{thm:Lorenzini}, but it seems the complexity of the task explodes substantially.
For example, $K_{n_1,\ldots,n_r}$ has exactly $|\{n_1,\ldots,n_r\}|+1$ distinct non-zero Laplacian eigenvalues, but once $r\geq 3$, the exponent of $\K(K_{n_1,\ldots,n_r})$ depends on the number theoretic property of the $n_i$'s and their multiplicities \cite[Theorem~1]{JNR_Critical}.
This already makes the classification not entirely straightforward to state for complete multipartite graphs.

On the other hand, unlike the case of ordinary graphs, where the notion of strong regularity implies having exactly two distinct Laplacian eigenvalues, the suggested notions of {\em (very) strongly regular signed graphs} \cite{Stanic3, zaslavsky2} have regularities only slightly weaker than Proposition~\ref{prop:SG_para}, but they can have more than two distinct Laplacian eigenvalues.
So it would be interesting to explore their critical groups in future works.

Finally, we believe the approach in Section~\ref{sec:pairing} has much to be developed.
Starting with Theorem~\ref{thm:tail_heavy}, it might be possible to formulate similar/stronger algebraic statements based on other types of group elements and their corresponding graphical notions.
Turning to the combinatorial side, one could try to find more robust constructions of orthogonal subsets of edges (or other relevant graphical notions), hence derive other bounds on the maximum size of orthogonal subsets.
As a particular question, in almost all examples we know (other than the exceptional cases in Theorem~\ref{thm:main}, a few small cases, and the family of complete tripartite graphs $K_{m,m,m}$'s), the critical group of a $(n,k,\lambda,\mu)$-SRG contains a subgroup isomorphic to $(\mathbb{Z}/n\mu\mathbb{Z})^r$ for some $r>1$.
We ask whether that is a general phenomenon for SRGs (excluding a concrete list of exceptions), and whether some strengthening of our approach can prove it.

\section*{Acknowledgements}
Both KH and CHY are grateful to Diocesan Boys' School for forging this mathematical connection during their secondary school days.
KH contributed while affiliated with California Institute of Technology, and was supported by the Summer Research Undergraduate Fellowship 2013; he wishes to acknowledge the great mentorship of Prof.\ Mohamed Omar.
CHY extended the results of KH, and was supported by the Trond Mohn Foundation project ``Algebraic and Topological Cycles in Complex and Tropical Geometries'' at the University of Oslo.
CHY thanks Richard Wilson and his book {\em A Course in Combinatorics} for inspiring his interest on SRGs, Dino Lorenzini for pointing out \cite{Lorenzini}, Matt Baker for the discussion on monodromy pairings, and Jasper Lee and Henry Tsang for their encouragements.
Both authors thank the anonymous referee for the helpful comments, especially for suggesting them to look into the monodromy pairing and its potential applications.

\vspace{-2mm}

\bibliographystyle{plain}

\bibliography{SRG}

\begin{thebibliography}{10}

\bibitem{BS_Pairing}
Matthew Baker and Farbod Shokrieh.
\newblock Chip-firing games, potential theory on graphs, and spanning trees.
\newblock {\em J. Combin. Theory Ser. A}, 120(1):164--182, 2013.

\bibitem{BL_Pairing}
Siegfried Bosch and Dino Lorenzini.
\newblock Grothendieck's pairing on component groups of {J}acobians.
\newblock {\em Invent. Math.}, 148(2):353--396, 2002.

\bibitem{SRG_database}
Andries~E. Brouwer.
\newblock Parameters of strongly regular graphs.
\newblock Online database,
  \url{https://www.win.tue.nl/~aeb/graphs/srg/srgtab.html}.

\bibitem{CSX_Paley}
David~B. Chandler, Peter Sin, and Qing Xiang.
\newblock The {S}mith and critical groups of {P}aley graphs.
\newblock {\em J. Algebraic Combin.}, 41(4):1013--1022, 2015.

\bibitem{SandpileBook}
Scott Corry and David Perkinson.
\newblock {\em Divisors and sandpiles}.
\newblock American Mathematical Society, Providence, RI, 2018.

\bibitem{DDEMPSV}
Joshua~E. Ducey, David~L. Duncan, Wesley~J. Engelbrecht, Jawahar~V. Madan, Eric
  Piato, Christina~S. Shatford, and Angela Vichitbandha.
\newblock Critical group structure from the parameters of a strongly regular
  graph.
\newblock {\em J. Combin. Theory Ser. A}, 180:105424, 20, 2021.

\bibitem{DGW_Rook}
Joshua~E. Ducey, Jonathan Gerhard, and Noah Watson.
\newblock The {S}mith and critical groups of the square rook's graph and its
  complement.
\newblock {\em Electron. J. Combin.}, 23(4):Paper 4.9, 19, 2016.

\bibitem{Evans}
Rhys~J. Evans.
\newblock Bounds for regular induced subgraphs of strongly regular graphs,
  2022.
\newblock https://arxiv.org/abs/2202.03700.

\bibitem{rA}
Michael Faux and S.~James {Gates, Jr.}
\newblock Adinkras: A graphical technology for supersymmetric representation
  theory.
\newblock {\em Phys. Rev. D (3)}, 71:065002, 2005.

\bibitem{AGT}
Chris Godsil and Gordon Royle.
\newblock {\em Algebraic graph theory}, volume 207 of {\em Graduate Texts in
  Mathematics}.
\newblock Springer-Verlag, New York, 2001.

\bibitem{GS_clique}
Gary R.~W. Greaves and Leonard~H. Soicher.
\newblock On the clique number of a strongly regular graph.
\newblock {\em Electron. J. Combin.}, 25(4):Paper No. 4.15, 15, 2018.

\bibitem{harary}
Frank Harary.
\newblock On the notion of balance of a signed graph.
\newblock {\em Michigan Math. J.}, 2:143--146 (1955), 1953/54.

\bibitem{HouTangWang2}
Yaoping Hou, Zikai Tang, and Dijian Wang.
\newblock On signed graphs with just two distinct {L}aplacian eigenvalues.
\newblock {\em Appl. Math. Comput.}, 351:1--7, 2019.

\bibitem{Huang}
Hao Huang.
\newblock Induced subgraphs of hypercubes and a proof of the sensitivity
  conjecture.
\newblock {\em Ann. of Math. (2)}, 190(3):949--955, 2019.

\bibitem{Iga_Clifford}
Kevin Iga.
\newblock Adinkras: {G}raphs of {C}lifford {A}lgebra {R}epresentations,
  {S}upersymmetry, and {C}odes.
\newblock {\em Adv. Appl. Clifford Algebr.}, 31(5):Paper No. 76, 2021.

\bibitem{IKKY}
Kevin Iga, Caroline Klivans, Jordan Kostiuk, and Chi~Ho Yuen.
\newblock Eigenvalues and critical groups of {A}dinkras, 2022.
\newblock https://arxiv.org/abs/2202.02821.

\bibitem{JNR_Critical}
Brian Jacobson, Andrew Niedermaier, and Victor Reiner.
\newblock Critical groups for complete multipartite graphs and {C}artesian
  products of complete graphs.
\newblock {\em J. Graph Theory}, 44(3):231--250, 2003.

\bibitem{ChipBook}
Caroline~J. Klivans.
\newblock {\em The mathematics of chip-firing}.
\newblock Discrete Mathematics and its Applications (Boca Raton). CRC Press,
  Boca Raton, FL, 2019.

\bibitem{LZ_Prym}
Yoav Len and Dmitry Zakharov.
\newblock Kirchhoff's theorem for {P}rym varieties.
\newblock {\em Forum Math. Sigma}, 10:Paper No. e11, 2022.

\bibitem{Lorenzini}
Dino Lorenzini.
\newblock Smith normal form and {L}aplacians.
\newblock {\em J. Combin. Theory Ser. B}, 98(6):1271--1300, 2008.

\bibitem{Lorenzini2}
Dino~J. Lorenzini.
\newblock A finite group attached to the {L}aplacian of a graph.
\newblock {\em Discrete Math.}, 91(3):277--282, 1991.

\bibitem{TS_Polar}
Venkata Raghu~Tej Pantangi and Peter Sin.
\newblock Smith and critical groups of polar graphs.
\newblock {\em J. Combin. Theory Ser. A}, 167:460--498, 2019.

\bibitem{RT_Signed}
Victor Reiner and Dennis Tseng.
\newblock Critical groups of covering, voltage and signed graphs.
\newblock {\em Discrete Math.}, 318:10--40, 2014.

\bibitem{Stanic3}
Zoran Stani\'{c}.
\newblock On strongly regular signed graphs.
\newblock {\em Discrete Appl. Math.}, 271:184--190, 2019.

\bibitem{Stanic2}
Zoran Stani\'{c}.
\newblock Spectra of signed graphs with two eigenvalues.
\newblock {\em Appl. Math. Comput.}, 364:124627, 9, 2020.

\bibitem{DH_2LE}
Edwin~R. van Dam and Willem~H. Haemers.
\newblock Graphs with constant {$\mu$} and {$\overline\mu$}.
\newblock volume 182, pages 293--307. 1998.
\newblock Graph theory (Lake Bled, 1995).

\bibitem{ACIC}
J.~H. van Lint and R.~M. Wilson.
\newblock {\em A course in combinatorics}.
\newblock Cambridge University Press, Cambridge, second edition, 2001.

\bibitem{zaslavsky}
Thomas Zaslavsky.
\newblock Signed graphs.
\newblock {\em Discrete Applied Mathematics}, 4:47--74, 1982.

\bibitem{zaslavsky2}
Thomas Zaslavsky.
\newblock Matrices in the theory of signed simple graphs.
\newblock In {\em Advances in discrete mathematics and applications: {M}ysore,
  2008}, volume~13 of {\em Ramanujan Math. Soc. Lect. Notes Ser.}, pages
  207--229. Ramanujan Math. Soc., Mysore, 2010.

\end{thebibliography}

\end{document}